\documentclass[11pt]{amsart}

\usepackage{url}
\usepackage{hyperref}
\usepackage{amsmath}
\usepackage{amssymb}
\usepackage{amscd}
\usepackage{manfnt}
\usepackage{enumerate}
\usepackage[all]{xy}
\usepackage{graphicx}
\usepackage{MnSymbol}
\usepackage[a4paper]{geometry}

\usepackage{tikz}
\usetikzlibrary{intersections,patterns,matrix}
\definecolor{linkcolor0}{rgb}{0.85, 0.15, 0.15}
\definecolor{linkcolor1}{rgb}{0.15, 0.15, 0.85}
\usepackage{pgfplots}

\usepackage{comment}
\usepackage[colorinlistoftodos]{todonotes}

\numberwithin{equation}{section}

\newtheorem{theorem}{Theorem}[section]
\newtheorem*{theorem*}{Theorem}
\newtheorem{proposition}[theorem]{Proposition}
\newtheorem{lemma}[theorem]{Lemma}
\newtheorem{corollary}[theorem]{Corollary}

\theoremstyle{definition}
\newtheorem{definition}[theorem]{Definition}
\newtheorem{example}[theorem]{Example}

\newtheorem{convention}[theorem]{Convention}

\theoremstyle{remark}
\newtheorem{remark}[theorem]{Remark}
\newtheorem*{ack}{Acknowledgments}

\def\R{\mathbb R}
\def\C{\mathbb C}
\def\Q{\mathbb Q}

\def\wt#1{\widetilde{#1}}

\def\cV{\mathcal{V}}
\def\cW{\mathcal{W}}

\def\ol#1{\overline{#1}}

\DeclareMathOperator\lk{lk}

\DeclareMathOperator\re{Re}

\DeclareMathOperator\sign{sign}

\title{Limits of the Tristram--Levine signature function}
\author{Maciej Borodzik}

\address{Institute of Mathematics, University of Warsaw, ul. Banacha 2,
02-097 Warsaw, Poland}
\email{mcboro@mimuw.edu.pl}
\author{Jakub Zarzycki}
\address{Institute of Mathematics, University of Warsaw, ul. Banacha 2,
02-097 Warsaw, Poland}
\email{jk.zarzycki@student.uw.edu.pl}

\makeatletter
\@namedef{subjclassname@2010}{\textup{2010} Mathematics Subject Classification}
\makeatother
\begin{document}

\begin{abstract}
We show that under a precise condition on the single variable Alexander polynomial, the limit at $1$ of the
Tristram--Levine signature of a link is determined by the linking matrix.
\end{abstract}

\maketitle

\section{Introduction}
Let $L= L_1\cup\dots\cup L_r$ be a link. Recall that the Tristram--Levine signature of $L$
is defined as:
\begin{equation}\label{eq:sigma}
\sigma_L(z)=\sign((1-z)S+(1-\ol{z})S^T)),
\end{equation}
where $S$ is a Seifert matrix for $L$, $z$ is a complex number with modulus $1$, and $\sign$ denotes the signature of a hermitian matrix. 
Levine--Tristram signature was introduced by Tristram \cite{Tristram} and Levine \cite{Levine_sig} as a generalization of the Murasugi signature \cite{Murasugi}. Levine--Tristram signature is an important link invariant, used e.g. to obstruct concordance of links.

By \eqref{eq:sigma}, the value of $\sigma_L(1)$ is zero. Nevertheless, if $r>1$, the limit
\[\sigma^1:=\lim_{\substack{t\to 0\\t\neq 0}}\sigma_L(e^{it})\]
can be non-zero.
Surprisingly, not much is known about the topological interpretation of $\sigma^1$. In \cite[Theorem 2.1]{GilmerLivingston} it is proved
that $|\sigma^1|\le r-1$; see also \cite[Remark 2.2]{Conway_survey}. This statement generalizes the well-known fact
that $\sigma^1=0$ if $L$ is a knot.

Before we state the main theorem, recall that
the \emph{linking matrix} of an $r$-component link $L$ is an
$r\times r$ matrix $A_L=\{a_{ij}\}$, where
\begin{equation}\label{eq:aij} a_{ij} = \begin{cases}
    \lk(L_i, L_j) &\textrm{if} \quad i \neq j \\
    -\lk(L_i, L \setminus L_i) & \textrm{otherwise.}
\end{cases}
\end{equation}
The following is the main theorem of this short article.
\begin{theorem}\label{thm:main}
  Let $L$ be an $r$-component link.
Let $\Delta_L(t)$ be the single variable Alexander polynomial. Suppose that $\Delta_L(t)$ is non-zero. 
If $(t-1)^r$ does not divide $\Delta_L(t)$, then $\sigma^1$ is the signature of the linking matrix of $L$ 
\end{theorem}
\begin{remark}
  Our definition of the linking matrix is different than the one in LinkInfo \cite{linkinfo}, where the linking matrix is supposed
  to have zeros on the diagonal. Our definition is maybe less natural, but has a more geometric meaning; see Lemma~\ref{lem:small_seifert} below. 
  Our convention is used e.g. in connection with Hosokawa theorem \cite{Hosokawa,Turaev}.
\end{remark}

Theorem~\ref{thm:main} is proved in Section~\ref{sec:proof_of_main} after we give preparatory results in Sections~\ref{sec:linking} -- \ref{sec:eigen}. Section~\ref{sec:examples} contains several examples. 
\begin{remark}
  The methods of the present paper can be used to generalize Theorem~\ref{thm:main} to the case when $\Delta_L\equiv 0$
  or if $(t-1)^r$ divides $\Delta_L(t)$, leading to inequalities for $|\sigma^1-\sign A_L|$. These statements are more technical
  than Theorem~\ref{thm:main} and the proof is therefore, they are not discussed in the present paper.
\end{remark}

\begin{convention}
  Throughout the paper, we use the notation $\lim_{z\to 1}f(z)$ for $\lim_{\substack{z\to 1\\ |z|=1, z\neq 1}} f(z)$.
\end{convention}
\begin{ack}
  The authors would like to thank Chuck Livingston for fruitful discussions and to Anthony Conway and Wojciech Politarczyk for careful reading of the preliminary version of the manuscript.
  The first author is supported by the Polish Center of Science grant
  OPUS 2019/35/B/ST1/01120. 
\end{ack}
\section{Linking matrix}\label{sec:linking}

We introduce a useful terminology.
\begin{definition}
  A \emph{small linking  matrix} $H_L$ for $L$ is an $(r-1)\times (r-1)$
  principal minor of $A_L$.
\end{definition}
\begin{lemma}\label{lem:small}
  The congruence class of $H_L$ is well-defined, that is, does not depend on the particular choice of a minor. The signature
  of $H_L$ is equal to the signature of $A_L$. The dimension of the
  kernel of $H_L$ is one less than the dimension of $A_L$.
\end{lemma}
\begin{proof}
  With the notation of \eqref{eq:aij}, we have $\sum_i a_{ij}=0$ and $\sum_j a_{ij}=0$. Therefore, the linking matrix is of the form
  $$ A_L = \begin{pmatrix}
  a_{1, 1} & a_{1, 2} &  \dots & a_{1, r-1} & -\sum_j a_{1, j} \\
  a_{2, 1} & a_{2, 2} & \dots & a_{2, r-1} & -\sum_j a_{2, j} \\
  \vdots \\
  -\sum_i a_{i, 1} & -\sum_i a_{i, 2} & \dots & -\sum_i a_{i, r-1} & \sum_{i, j} a_{i, j} \\
  \end{pmatrix}.$$

  With this formula, any row is a linear combination of other rows. Removing a linearly dependent row and the same column from a symmetric
  matrix does not change the signature and the rank. 
\end{proof}

A small linking matrix has a geometric interpretation.
\begin{lemma}\label{lem:small_seifert}
  Let $\Sigma$ be a Seifert surface for $L$ and let
  $H$ be the image of the inclusion induced map $H_1(\partial\Sigma;\C)\to H_1(\Sigma;\C)$. Then, the Seifert form restricted to $H$ is congruent to a small matrix $H_L$. In particular the Seifert form restricted to $H$ is symmetric.
\end{lemma}
\begin{proof}
  Recall that $L_1,\dots,L_r$ denote the components of the link.
  Denote by  $[L_1],\dots,[L_{r}]$ the classes they represent in
  $H_1(\Sigma;\C)$. These classes span the space $H$, and are subject to the relation $[L_1]+\dots+[L_r]=0$.

  As $L_i$ and $L_j$ for $i\neq j$ are disjoint, the Seifert form $S([L_i],[L_j])$ is equal to the linking number
  of $L_i$ and $L_j$. Due to the relation $[L_1]+\dots+[L_r]=0$, we infer that
  $S([L_i],[L_i])=-\lk(L_i,L_1\cup\dots\cup L_{i-1}\cup\dots\cup L_r)$.

  To conclude the proof, we choose a basis $[L_1],\dots,[L_{r-1}]$ of $H$. In this basis $S$ is represented by the $(r-1)\times (r-1)$ minor
  of the linking matrix $A_L$.
\end{proof}
\begin{example}\label{ex:two_component}
  Suppose $L$ is a two component link. If $\ell$ denotes the linking number of the two components, then the linking matrix is $\begin{pmatrix} -\ell & \ell \\ \ell & -\ell \end{pmatrix}$.
  The small linking matrix is equal to $(-\ell)$. In particular, the signature of the small linking matrix is \emph{minus} the sign of the linking number.
\end{example}
\section{Review of Hodge numbers}\label{sec:hodge}
Hodge numbers are objects related to so called hermitian variation structures. 
These were introduced by N\'emethi \cite{Nem_real} as an algebraic framework for describing homological structures
associated with the Milnor fibration of an isolated hypersurface singularity.
We quickly review the terminology following \cite{Nem_real}. We restrict our attention to hermitian variation structures with sign $\epsilon=-1$. In what follows, the star denotes the dual space or the dual map. In matrix notation, the star means transposition followed by complex conjugation.

A \emph{hermitian variation structure}, in short: an HVS, is a quadruple $(U,b,h,V)$, where $U$ is a finite dimensional vector space over $\C$,
$b\colon U\to U^*$ is a map inducing a skew-Hermitian pairing on $U$ (that is $\overline{b}^*=-b$), $h\colon U\to U$ is $b$-orthogonal and $V\colon U^*\to U$ is such that $\overline{V^*}=V\circ\overline{h^*}$ and $V\circ b=h-I$. Throughout the paper we will assume that HVS are \emph{simple}, that is, $V$ is invertible. In this case, $V$ determines both $b$ and $h$ via formulae
\begin{equation}\label{eq:v_determines}
  b=(V^*)^{-1}-V^{-1},\ h=V(V^*)^{-1}.
\end{equation}

Two HVS $(U,b,h,V)$ and $(U',b',h',V')$ are isomorphic if there
is a map $\phi\colon U\to U'$ intertwining $b,h,V$ with $b',h',V'$
respectively. This means that $b=\ol{\phi}^* b'\phi$, $h=\phi^{-1}h'\phi$ and $V=\phi^{-1}V'(\ol{\phi}^*)^{-1}$; see \cite[Definition 2.5]{Nem_real}. It is worth to stress that in matrix notation, $h$ transforms by a conjugation, and $b$ and $V$ transform by a congruence. We note that there is rather obvious notion of a direct
sum of two HVS.

Basic examples of HVS were given by N\'emethi in \cite[Section 2]{Nem_real}.
For any integer $k$, a unit complex number $\lambda$ and a sign choice $u\in\{1,-1\}$, there exists
an HVS $\cW^k_\lambda(u)$, such that $\dim U=k$ and $h$ is a single Jordan block with eigenvalue $\lambda$. The structures
$\cW^k_\lambda(+1)$ and $\cW^k_\lambda(-1)$ are distinguished by the signature of the pairing $b$.
In this paper we need an explicit form of the structure $\cW^1_1(u)$:
\begin{equation}\label{eq:v1}
\cW^1_1(u)=(\C,0,1,u);
\end{equation}
compare \cite[Example 2.7, item 7]{Nem_real}.

On the other hand, for any integer $\ell$ and a complex number $\mu$ with $|\mu|\in(0,1)$, there exists
an HVS $\cV^{2\ell}_\mu$. For this structure, 
$h$ is a sum of Jordan blocks of size $k$: one with eigenvalue $\mu$, another one with eigenvalue $\overline{\mu}^{-1}$.
\begin{remark}\label{rem:for_future}
For future use, we note that all the structures $\cV^{2\ell}_\mu$ and $\cW^k_\lambda(u)$ for $\lambda\neq 1$
have non-degenerate pairing $b$. The pairing $b$ associated to each of the structures $\cW^k_1(u)$,
has one-dimensional kernel.
\end{remark}
The classification theorem tells us that each simple HVS can be presented as a direct sum of
basic structures. More precisely, we have the following result.
\begin{theorem}[see \expandafter{\cite[Theorem 2.9]{Nem_real}}]\label{thm:class}
  Suppose $\cV$ is a simple HVS. There exists uniquely defined
  non-negative integers $p^k_\lambda(u)$ and $q^\ell_\mu$ (with $k,\ell,\lambda,\mu,u$ as above)
  such that $\cV$ is isomorphic to the sum
  \begin{equation}\label{eq:sum_of_all}
    \cV=\bigoplus_{k,\lambda,u} p^k_{\lambda}(u)\cW^k_\lambda(u)\oplus\bigoplus_{\ell,\mu}q^\ell_\mu\cV^{2\ell}_{\mu},
  \end{equation}
  where the symbol $s\cV$ for a non-negative integer $s$ should be interpreted as the direct sum of $s$ copies of the structure $\cV$.
\end{theorem}
We call the integers $p^k_\lambda(u)$ and $q^\ell_\mu$ the 
\emph{Hodge numbers} associated with the HVS $\cV$.
We will refer to the structures $\cW^k_\lambda(u)$ and $\cV^{2\ell}_\mu$ as \emph{basic structures}.

\smallskip
In \cite[Section 3.1]{Hodge_type} there was defined an HVS for a link in $S^3$. The following construction
is a special case of that construction.
\begin{proposition}\label{prop:link_hvs}
Suppose $L\subset S^3$ is a link whose Seifert matrix is S-equivalent over $\R$ to a non-degenerate matrix
$S$. Then there exists a well-defined HVS for the link, such that $V=(S^T)^{-1}$. The pairing $b$ of this structure is adjoint to the form $S-S^T$ and $h=(S^T)^{-1}S$.
\end{proposition}
The Hodge numbers associated with the link determine its Alexander polynomial and the Tristram--Levine signature, see \cite[Section 4]{Hodge_type}. Hodge numbers for links satisfy the relations $p^k_{\lambda}(u)=p^k_{\overline{\lambda}}(u)$, $q^\ell_{\mu}=q^\ell_{\overline{\mu}}$.

\section{The part of HVS with eigenvalue $1$}\label{sec:eigen}
The next result is due to Keef \cite{Keef}. We refer to \cite[Section 4]{BZ} for a more detailed
exposition of Keef's results in the present context.
\begin{lemma}\label{lem:alex_non_zero}
  Let $L$ be a link. The one-variable Alexander polynomial of $L$ is non-zero if and only if any Seifert matrix of $L$
  is S-equivalent (over $\Q$) to an invertible matrix.
\end{lemma}
\begin{proof}
  Let $S$ be a Seifert matrix of $L$. By Keef \cite[Proposition 3.1]{Keef} we know that $S$ is S-equivalent over $\Q$ to a block matrix $S_{ndeg}\oplus S_0$,
  where $S_{ndeg}$ is invertible and $S_0$ is a matrix with all entries zero. 
  The one-variable Alexander polynomial of $L$ is given by $\det(tS-S^T)$, which
  is zero if $S_0$ has positive dimension.

  Note that the rank of $S_0$ is the rank of the torsion-free part of the Alexander module of $L$ over $\Q[t,t^{-1}]$, hence it does not
  depend on the particular choice of the Seifert matrix within
  its S-equivalence class.
\end{proof}
From now on, we will assume that $L$ is an $r$-component link whose  one-variable Alexander polynomial is non-zero.
Let $S$ be an invertible matrix which is S-equivalent to a Seifert matrix of $L$. We let $\cV_L=(U_L,b_L,h_L,V_L)$
be the HVS associated with $L$. In the decomposition \eqref{eq:sum_of_all} we group terms with $\lambda=1$ and separately all other terms
so as to obtain a decomposition $\cV_L=\cV_{=1}\oplus \cV_{\neq 1}$; 
compare \cite[Section 3.3]{BoNe_Last}.
Accordingly, we write:
\begin{equation}\label{eq:decomp}
U_L=U_{\neq 1}\oplus U_{=1},\ b_L=b_{\neq 1}\oplus b_{=1},\ h_L=h_{\neq 1}\oplus h_{=1},\ V_L=V_{\neq 1}\oplus V_{=1}.\end{equation}
Note that the key property is that all the eigenvalues of $h_{\neq 1}$ are different than $1$ and $h_{=1}$ is a sum of Jordan
blocks with eigenvalue $1$.

Given \eqref{eq:decomp}, as $V_L=(S^T)^{-1}$, changing $S$ by a congruence if needed,
we can decompose the matrix $S$ as a block sum
\begin{equation}\label{eq:decompose_S}
  S=S_{\neq 1}\oplus S_{=1};
\end{equation}
see \cite[Proposition 3.3.10(d)]{BoNe_Last}.
The following observation follows immediately
from the definition of $U_{=1}$.
\begin{lemma}
  Let $p^k_\lambda(u)$, $q^\ell_\mu$ be the Hodge numbers associated with $\cV_L$. Then
  \begin{equation}\label{eq:rank}
  \dim U_{=1}=\sum_{k,u} kp^k_1(u).\end{equation}
\end{lemma}
Our next aim is to relate the structure $\cV_{=1}$ with
the linking matrix of the link $L$. For links
of singularities the result we are going to prove can
be found in \cite[Section 3]{Neumann}, even though it is phrased in a different language.

The case of links with possibly degenerate Seifert matrix requires a technical result. Before we
state it, we recall that a \emph{row extension} replaces a matrix $S_1$ by a matrix
\begin{equation}\label{eq:row_ext}
S_2 = \begin{pmatrix}
S_1 & 0 & 0  \\
\xi & 0 & 0 \\
0 & 1 & 0
\end{pmatrix},
\end{equation}
where $\xi$ is a row vector. A \emph{row contraction} is the inverse operation.
 Column extensions and contractions are defined analogously.
There are several conventions regarding these definitions, leading to equivalent formula with different
shape; we adopt the convention that is used
e.g. by Levine and Murasugi; see \cite{Levine_role,Murasugi}.
We also recall that two square matrices are S-equivalent if one can pass
from one to another by a sequence of congruences, row and column extensions, and row and column contractions.

\begin{lemma}\label{lem:preserve}
  Let $U_1,U_2$ be two vector spaces over $\R$, with an embedding $\iota : U_1\to U_2$. Let $S_i\colon U_i\to U_i^*$, $i=1,2$ two maps such that in some choice of basis
  the matrix representing $S_2$ is a row or column extension of a matrix representing $S_1$. Then
  $\iota$ induces an isomorphism $\iota'\colon \ker (S_1-S_1^T)\xrightarrow{\simeq}\ker (S_2-S_2^T)$.
  The isomorphism $\iota'$ induces an isometry of forms $(S_1+S_1^T)$ and $(S_2+S_2^T)$ restricted to $\ker(S_1-S_1^T)$, respectively
  $\ker(S_2-S_2^T)$.
\end{lemma}
\begin{proof}
Let $n=\dim U_1$.
We restrict to the case when $S_2$ is a row extenstion of $S_1$, the case of the column extension is analogous. In that case,
by \eqref{eq:row_ext} we have
\begin{equation}\label{eq:s12}S_2 - S_2^T = \begin{pmatrix}
S_1 - S_1^T & -\xi^T & 0 \\
\xi & 0 & -1 \\
0 & 1 & 0
\end{pmatrix}.
\end{equation}
For the $(n+2)$-dimensional vector $(u_1,u_2,u_3)$ (with $u_1$ being an $n$-dimensional vector,
$u_2$ and $u_3$ being numbers)
\begin{equation}\label{eq:s2} (S_2 - S_2^T) \begin{pmatrix}u_1 \\ u_2 \\ u_3 \end{pmatrix} = 
\begin{pmatrix}
(S_1 - S_1^T)u_1 - \xi^Tu_2 \\ 
\xi\cdot u_1 - u_3 \\
u_2
\end{pmatrix}.
\end{equation}
Here $\xi\cdot u_1$ is a scalar product. Note that, on the contrary, $\xi^Tu_2$ (without dot) is a column vector.
The map $\iota'$ is defined as $\iota'(u)=(u,0,\xi\cdot u)$.
A straightforward calculation involving \eqref{eq:s2} shows that $\iota'$ takes $\ker (S_1-S_1^T)$
onto $\ker (S_2-S_2^T)$.

Our aim is to show that for any $u,u'\in U_1$, $\iota(u)^T(S_2+S_2^T)\iota(u')=u^T(S_1+S_1^T)u'$. To this end, using \eqref{eq:s12},
we write
\[S_2+S_2^T=\begin{pmatrix} S_1+S_1^T & \xi^T & 0 \\ \xi & 0 & 1 \\ 0 & 1 & 0\end{pmatrix}.\]
Hence, 
\[
  \iota(u)^T(S_2+S_2^T)\iota(u')=
  \begin{pmatrix} u \\ 0\\ \xi\cdot u\end{pmatrix}^T\begin{pmatrix} S_1+S_1^T & \xi^T & 0 \\ \xi & 0 & 1 \\ 0 & 1 & 0\end{pmatrix}
\begin{pmatrix} u' \\ 0\\ \xi\cdot u'\end{pmatrix}=
u^T(S_1+S_1^T) u'.
\]
\end{proof}
\begin{corollary}\label{cor:final}
  Let $\wt{S}$ be a Seifert matrix for an $r$-component link $L$.
  Suppose $\wt{S}$ is $S$-equivalent
  to a non-degenerate matrix $S$. Then:
  \begin{itemize}
    \item The spaces $H=\ker (S-S^T)$ and $\wt{H}=\ker(\wt{S}-\wt{S}^T)$ are
      isomorphic.
    \item The forms $S+S^T$ and $\wt{S}+\wt{S}^T$ restricted to $H$ and $\wt{H}$ respectively have the same rank and signature.
  \end{itemize}
\end{corollary}
\begin{proof}
  The matrices $S$ and $\wt{S}$ are related by a sequence of congruences and row/column extensions and row/column contractions. By Lemma~\ref{lem:preserve} row/column extensions preserve $\ker (S-S^T)$ and the form $S+S^T$ restricted to the kernel. Congruences clearly preserves $\ker(S-S^T)$ and the form $S+S^T$.
\end{proof}

Suppose now $L$ is a link whose Alexander polynomial is not identically zero. Let $S$ be an invertible matrix S-equivalent
to a Seifert matrix $\wt{S}$ of $L$ and define $\cV_L=(U_L,b_L,h_L,V_L)$ to be the HVS associated with $L$. Recall the decomposition
$\cV_L=\cV_{=1}\oplus \cV_{\neq 1}$
\begin{lemma}\label{lem:condition_rev}
  Supose the Alexander polynomial of $L$ is non-zero. The following conditions are equivalent.
  \begin{itemize}
  \item[(i)]  The Hodge numbers $p^k_1(\epsilon)$ are zero for $k>1$;
  \item[(ii)] $U_{=1}$ is equal to $\ker b$;
  \item[(iii)]  $(t-1)^r$ does not divide $\Delta(t)$.
\end{itemize}
\end{lemma}
\begin{proof}
  Each of the structures $\cW^k_1(u)$ has the property that $\ker b$ is one dimensional. All other basic structures,
  that is, $\cW^k_\lambda(u)$ for $\lambda\neq 1$ or $\cV^{2\ell}_\mu$, have non-degenerate $b$; see Remark~\ref{rem:for_future}
  for an explicit form of these structures in \cite[Section 2]{Nem_real}. This shows that
  \[
    \dim\ker b_L=\sum p^k_1(u),\ \ker b_L\subset U_{=1}.
  \]
  By \eqref{eq:rank} we immediately show that (i) is equivalent to (ii).

  The Alexander polynomial $\Delta(t)$ is equal to $\det(h_L-tI)$ up
  to multiplication by a unit in $\C[t,t^{-1}]$. In particular, the
  maximal integer $\ell$ such that $(t-1)^\ell$ divides $\Delta(t)$
  is equal to the dimension of the eigenspace of $h_L$ corresponding to eigenvalue $1$. That is, $\ell=\dim U_{=1}$.
  To prove equivalence of (ii) and (iii), it is enough to show that $\dim\ker b_L=r-1$.

  By Proposition~\ref{prop:link_hvs}, $\ker b_L=\ker(S^T-S)$.  By the first part of Corollary~\ref{cor:final}, we have
  $\dim\ker S^T-S=\dim\ker \wt{S}^T-\wt{S}$. The form $\wt{S}^T-\wt{S}$
  is the intersection form on $H_1(\Sigma,\C)$. As $\Sigma$ is an oriented surface with $r$ boundary components $\dim\ker (\wt{S}^t-\wt{S})=r-1$. 
\end{proof}

\section{Proof of Theorem~\ref{thm:main}}\label{sec:proof_of_main}
Given the translation between the assumptions of Theorem~\ref{thm:main} and Hodge numbers we can give the proof of Theorem~\ref{thm:main}.

\begin{lemma}\label{lem:diagonalize}
  Let $\cV=(U,b,h,V)$ be a HVS.
  Suppose $p^k_1(u)=0$ for $k>1$. Then, with $S=(V^{-1})^T$, the form $S+S^T$ on
  $\ker b$ has signature $p^k_1(+1)-p^k_1(-1)$.
\end{lemma}
\begin{proof}
  Consider the decomposition $\cV=\cV_{=1}\oplus\cV_{\neq 1}$. By Lemma~\ref{lem:condition_rev}
  we have $\ker b=U_{=1}$. Therefore, $S+S^T$ on $\ker b$ is precisely $S_{=1}+S_{=1}^T$.

  Note that $p^k_1(u)=0$ for $k>1$ implies that $\cV_{=1}$ is a sum
  of $p^1_1(+1)$ copies of $\cW_1^1(+1)$ and $p^1_1(-1)$ copies of $\cW_1^1(-1)$. Thus,
  $V_{=1}$ is diagonal with signature $p^1_1(+1)-p^1_1(-1)$. Therefore,
$S_{=1}+S_{=1}^T$ is diagonal with the same signature.
\end{proof}
\begin{lemma}\label{lem:limit_to_signature}
  Let $L$ be link, $\cV_L$ be the HVS associated to it and suppose
  $p^k_1(u)=0$ for $k>1$.
  The limit of the signature function $\lim_{z\to 1} \sigma_L(z)$ is equal to $p^1_1(+1)-p^1_1(-1)$.
\end{lemma}
\begin{proof}[Sketch of proof]
  The statement follows immediately from \cite[Proposition 4.14]{Hodge_type}. For the sake of completeness we recall the main elements of the proof. 

  The signature function $z\mapsto\sigma(z)$ can be associated with
  any simple HVS via the formula resembling \eqref{eq:sigma}: $\sigma(z)=\sign((1-z)({V^{-1}})^{T}+(1-\overline{z})V^{-1})$. With this approach
  the signature is additive with respect direct sums of HVS. In particular, to prove Lemma~\ref{lem:limit_to_signature}, it is enough to understand the behavior of $\sigma(z)$ for structures $\cW^k_\lambda(u)$
  and $\cV^{2\ell}_\mu$. Now, for $\cV^{2\ell}_\mu$ we have vanishing
  $\sigma(z)$ by explicit calculations and for $\cW^k_{\lambda}(u)\oplus\cW^k_{\ol{\lambda}}(u)$ with $\lambda\neq\pm 1$, the limit at $1$ of $\sigma(z)$ can be shown to be zero. From these computations, and
  using symmetry $p^k_\lambda(u)=p^k_{\overline{\lambda}}(u)$, we show that 
  the only contribution to $\lim_{z\to 1}\sigma(z)$ can come from $\cW^k_1(u)$. In our situtation,
  only $\cW^1_1(u)$ is a summand of $\cV$. For this structure, by \eqref{eq:v1}, we conclude that $\sigma(z)=u$ for all $z\in S^1\setminus\{1\}$, in particular $\lim_{z\to 1}\sigma(z)=u$.
\end{proof}
Now we are in position to give a proof of the main theorem.
\begin{proof}[Proof of Theorem~\ref{thm:main}]
  Let $\Sigma$ be a Seifert surface for $L$ and let $\wt{S}$ be the Seifert matrix associated to it. Let $H_{\wt{S}}\subset H_1(\Sigma;\C)$ be the kernel of $\wt{S}-\wt{S}^T$. Then $H_{\wt{S}}$ can be identified with the image
  $H_1(\partial\Sigma;\C)\to H_1(\Sigma;\C)$.
  Next, let $S$ be an invertible matrix S-equivalent to $\wt{S}$.
  Write $H_S=\ker(S-S^T)$.
  We claim that the following numbers are equal:
  \begin{itemize}
    \item[(a)] The signature of the linking matrix of $L$;
    \item[(b)] The signature of the small linking matrix of $L$;
    \item[(c)] The signature of $\wt{S}+\wt{S}^T$ restricted
      to $\ker(\wt{S}-\wt{S}^T)$;
    \item[(d)] The signature of $S+S^T$ restricted
      to $\ker(S-S^T)$;
    \item[(e)] The difference $p^1_1(+1)-p^1_1(-1)$ for the Hodge
      numbers associated with $\cV_L$;
    \item[(f)] The limit signature $\sigma^1$.
  \end{itemize}
  Equality of (a) and (b) is Lemma~\ref{lem:small}. Equality of (b) and (c) is precisely Lemma~\ref{lem:small_seifert}. Equality of (c) and (d)
  follows from Corollary~\ref{cor:final}. Note that until this moment we do not use the assumption of Theorem~\ref{thm:main}. Now,
  the equality of (d) and (e) follows from Lemma~\ref{lem:diagonalize} which uses Theorem~\ref{thm:main} via the equivalence of items (i) and (iii) of Lemma~\ref{lem:condition_rev}. Finally, equality of (e) and (f)
  follows from Lemma~\ref{lem:limit_to_signature} (using again Condition~(iii) of Lemma~\ref{lem:condition_rev}).
\end{proof}

\section{Examples}\label{sec:examples}
\subsection{The Hopf link}
For the positive Hopf link we have $\lk(L_1, L_2) = 1$ so the linking matrix becomes 
\[A = \begin{pmatrix}
-1 & 1 \\
1 & -1
\end{pmatrix}\]
The signature of this matrix is $-1$ and so is the signature $\sigma^1$.
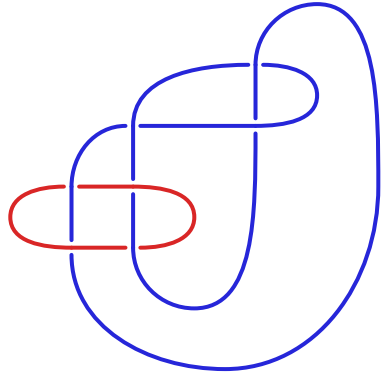
\begin{figure}
\begin{tikzpicture}[line width=1.5, line cap=round, line join=round, scale=0.5]
  \begin{scope}[color=linkcolor0]
    \draw (3.76, 4.58) .. controls (4.53, 4.58) and (5.37, 4.43) .. 
          (5.37, 3.77) .. controls (5.37, 3.17) and (4.65, 2.96) .. (3.95, 2.96);
    \draw (3.56, 2.96) .. controls (3.09, 2.96) and (2.61, 2.96) .. (2.14, 2.96);
    \draw (2.14, 2.96) .. controls (1.37, 2.96) and (0.53, 3.11) .. 
          (0.53, 3.77) .. controls (0.53, 4.37) and (1.25, 4.58) .. (1.94, 4.58);
    \draw (2.34, 4.58) .. controls (2.81, 4.58) and (3.28, 4.58) .. (3.76, 4.58);
  \end{scope}
  \begin{scope}[color=linkcolor1]
    \draw (2.14, 2.77) .. controls (2.14, 0.88) and (4.11, -0.26) .. 
          (6.18, -0.26) .. controls (8.57, -0.26) and (10.21, 2.04) .. 
          (10.21, 4.58) .. controls (10.21, 6.78) and (10.21, 9.42) .. 
          (8.60, 9.42) .. controls (7.71, 9.42) and (6.98, 8.70) .. (6.98, 7.81);
    \draw (6.98, 7.81) .. controls (6.98, 7.33) and (6.98, 6.86) .. (6.98, 6.39);
    \draw (6.98, 5.99) .. controls (6.98, 3.87) and (6.98, 1.35) .. 
          (5.37, 1.35) .. controls (4.48, 1.35) and (3.76, 2.07) .. (3.76, 2.96);
    \draw (3.76, 2.96) .. controls (3.76, 3.44) and (3.76, 3.91) .. (3.76, 4.38);
    \draw (3.76, 4.78) .. controls (3.76, 5.25) and (3.76, 5.72) .. (3.76, 6.19);
    \draw (3.76, 6.19) .. controls (3.76, 7.45) and (5.32, 7.81) .. (6.79, 7.81);
    \draw (7.18, 7.81) .. controls (7.87, 7.81) and (8.60, 7.60) .. 
          (8.60, 7.00) .. controls (8.60, 6.34) and (7.76, 6.19) .. (6.98, 6.19);
    \draw (6.98, 6.19) .. controls (5.97, 6.19) and (4.96, 6.19) .. (3.95, 6.19);
    \draw (3.56, 6.19) .. controls (2.74, 6.19) and (2.14, 5.44) .. (2.14, 4.58);
    \draw (2.14, 4.58) .. controls (2.14, 4.11) and (2.14, 3.63) .. (2.14, 3.16);
  \end{scope}
\end{tikzpicture}
\caption{Link L7a2. The orientation is chosen in such a way that the two components have linking number $-2$. Picture
drawn using \cite{snappy}.}\label{fig:l7a2}
\end{figure}

\subsection{Link L7a2}
The link L7a2, see Figure~\ref{fig:l7a2}, is a two-component link whose Seifert matrix is degenerate. Namely, according to LinkInfo \cite{linkinfo}
its Seifert matrix is
\[S=\begin{pmatrix}
0& 0& 0& -1& 0& 0& 0& 0& 0& 0& 0\\
0& 1& 0& 0& 0& 0& 0& -1& 0& 0& 0\\
0& -1& 1& 0& 0& 0& 0& 1& -1& 0& 0\\
0& 0& -1& 1& 0& 0& 0& 0& 0& 0& 0\\
0& 0& 0& -1& 1& 0& 0& 0& 1& -1& 0\\
0& 0& 0& 0& -1& 1& 0& 0& 0& 0& 0\\
0& 0& 0& 0& 0& -1& 1& 0& 0& 1& 0\\
0& 0& 0& 0& 0& 0& 0& 0& 1& 0& -1\\
0& 0& 0& 0& 0& 0& 0& 0& 0& 0& 0\\
0& 0& 0& 0& 0& 0& 0& 0& -1& 0& 1\\
0& 0& 0& 0& 0& 0& 0& 0& 0& 0& 0
\end{pmatrix}.\]
We have $\det(tS-S^T)=(3t^6-4t^5+3t^4)(t-1)$, and $3t^2-4t+3$ has two roots on $S^1\setminus\{1\}$, namely $\frac23\pm i\frac13\sqrt{5}$. The assumptions of Theorem~\ref{thm:main} are satisified.

The signature function is constant away from $z=\frac23\pm i\frac13\sqrt{5}$. Therefore, the limit signature at $1$ is equal to the signature at any point $z_0$ on $S^1\setminus\{1\}$, such that $\re z_0>\re z$.
With $z=\frac{4}{5}+i\frac{3}{5}$, the signature of the matrix $(1-z)S+(1-\overline{z})S^T$ can be computed using Sage \cite{sagemath}. It is equal to $1$.

The components of $L$ has linking number $-2$. Therefore, the linking matrix is $\begin{pmatrix} 2 & -2 \\ -2 & 2\end{pmatrix}$ and the small linking matrix is $(2)$. The signature of the small linking matrix
is $1$; this confirms Theorem~\ref{thm:main}.

\begin{figure}
\begin{tikzpicture}[line width=1.5, line cap=round, line join=round, scale=0.5]
  \begin{scope}[color=linkcolor0]
    \draw (3.24, 6.51) .. controls (1.67, 6.51) and (1.01, 4.68) .. 
          (1.01, 2.88) .. controls (1.01, 0.88) and (2.64, -0.75) .. 
          (4.64, -0.75) .. controls (6.46, -0.75) and (8.28, 0.06) .. (8.28, 1.67);
    \draw (8.28, 1.67) .. controls (8.28, 2.41) and (8.28, 3.15) .. (8.28, 3.90);
    \draw (8.28, 4.29) .. controls (8.28, 5.56) and (7.16, 6.51) .. (5.85, 6.51);
    \draw (5.85, 6.51) .. controls (5.11, 6.51) and (4.37, 6.51) .. (3.63, 6.51);
  \end{scope}
  \begin{scope}[color=linkcolor1]
    \draw (5.85, 4.09) .. controls (5.85, 2.79) and (6.81, 1.67) .. (8.08, 1.67);
    \draw (8.47, 1.67) .. controls (9.55, 1.67) and (10.70, 1.95) .. 
          (10.70, 2.88) .. controls (10.70, 3.88) and (9.43, 4.09) .. (8.28, 4.09);
    \draw (8.28, 4.09) .. controls (7.53, 4.09) and (6.79, 4.09) .. (6.05, 4.09);
    \draw (5.66, 4.09) .. controls (4.39, 4.09) and (3.43, 5.21) .. (3.43, 6.51);
    \draw (3.43, 6.51) .. controls (3.43, 7.67) and (3.65, 8.94) .. 
          (4.64, 8.94) .. controls (5.57, 8.94) and (5.85, 7.79) .. (5.85, 6.71);
    \draw (5.85, 6.32) .. controls (5.85, 5.58) and (5.85, 4.83) .. (5.85, 4.09);
  \end{scope}
\end{tikzpicture}
\caption{Link $L5a1$. Diagram drawn using SnapPy \cite{snappy}.}\label{fig:l5a1}
\end{figure}
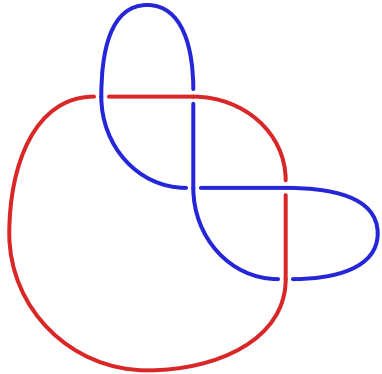
\subsection{L5a1 link}\label{sub:whiteh}
Consider the Whitehead link (L5a1 on the linkinfo list) depicted in Figure~\ref{fig:l5a1}.
This is a two-component link and the linking number of the two components is $0$. Thus, the linking matrix is the zero $2\times 2$ matrix and the
small linking matrix is $(0)$. In particular, the signature of the small linking matrix is zero.

According to LinkInfo \cite{linkinfo}, the Seifert matrix of L5a1 is
\[S = \begin{pmatrix}
1  & 0 & -1\\
-1 & 1 & 1\\
0  & 0 & -1
\end{pmatrix}.\]

The Alexander polynomial $\Delta(t) = \det(tS-S^T)=(t-1)^3$, therefore the assumptions of Theorem~\ref{thm:main} are not satisfied.

To calculate $\sigma^1$, we note that the signature function is constant away from the set of roots of the Alexander polynomial $\Delta$. As $\Delta$
has roots only at $t=1$, we infer that $\sigma^1=\sigma(-1)$. The latter signature is equal to
\[\sigma(-1)=\sign\begin{pmatrix} 2 & -1 & -1 \\ -1 & 2 & 1 \\ -1 & 1 & -2 \end{pmatrix}.\]
The above matrix has $2>0$ in the top left corner. The determinant of the top left $2\times 2$ minor is $3>0$ and the determinant of the whole $3\times 3$ matrix is $-8<0$. Hence,
$\sigma(-1)=1$ and so $\sigma^1=1$ is not equal to the signature of the small linking matrix. 
This example shows that the assumption of Theorem~\ref{thm:main} that $(t-1)^r$ does not divide the Alexander polynomial is necessary.

\bibliographystyle{plain}
\def\MR#1{}
\bibliography{research}

\begin{thebibliography}{10}

\bibitem{Hodge_type}
M.~Borodzik and A.~N\'{e}methi.
\newblock Hodge-type structures as link invariants.
\newblock {\em Ann. Inst. Fourier (Grenoble)}, 63(1):269--301, 2013.

\bibitem{BoNe_Last}
M.~Borodzik and A.~N\'{e}methi.
\newblock The {H}odge spectrum of analytic germs on isolated surface
  singularities.
\newblock {\em J. Math. Pures Appl. (9)}, 103(5):1132--1156, 2015.

\bibitem{BZ}
M.~Borodzik and J.~Zarzycki.
\newblock Real {S}eifert forms, {H}odge numbers and {B}lanchfield pairings.
\newblock 2019.
\newblock arXiv:1912.07690.

\bibitem{Conway_survey}
A.~Conway.
\newblock The {L}evine-{T}ristram signature: a survey, 2019.
\newblock arXiv:1903.04477.

\bibitem{snappy}
M.~Culler, N.. Dunfield, M.~Goerner, and J.~Weeks.
\newblock Snap{P}y, a computer program for studying the geometry and topology
  of $3$-manifolds.
\newblock available at \url{http://snappy.computop.org}.

\bibitem{GilmerLivingston}
P.~Gilmer and C.~Livingston.
\newblock Signature jumps and {A}lexander polynomials for links.
\newblock {\em Proc. Amer. Math. Soc.}, 144(12):5407--5417, 2016.

\bibitem{Hosokawa}
F.~Hosokawa.
\newblock On {$\nabla $}-polynomials of links.
\newblock {\em Osaka Math. J.}, 10:273--282, 1958.

\bibitem{Keef}
P.~Keef.
\newblock On the {$S$}-equivalence of some general sets of matrices.
\newblock {\em Rocky Mountain J. Math.}, 13(3):541--551, 1983.

\bibitem{Levine_sig}
J.~Levine.
\newblock Knot cobordism groups in codimension two.
\newblock {\em Comment. Math. Helv.}, 44:229--244, 1969.

\bibitem{Levine_role}
J.~Levine.
\newblock The role of the {S}eifert matrix in knot theory.
\newblock In {\em Actes du {C}ongr\`es {I}nternational des {M}ath\'{e}maticiens
  ({N}ice, 1970), {T}ome 2}, pages 95--98. 1971.

\bibitem{linkinfo}
C.~Livingston and A.~Moore.
\newblock Linkinfo: Table of link invariants.
\newblock URL: \url{linkinfo.math.indiana.edu}, accessed on December 3rd, 2020.

\bibitem{Murasugi}
K.~Murasugi.
\newblock On a certain numerical invariant of link types.
\newblock {\em Trans. Amer. Math. Soc.}, 117:387--422, 1965.

\bibitem{Nem_real}
A.~{N}{\'{e}}methi.
\newblock The real {S}eifert form and the spectral pairs of isolated
  hypersurface singularities.
\newblock {\em Compositio Math.}, 98(1):23--41, 1995.

\bibitem{Neumann}
W.~Neumann.
\newblock Invariants of plane curve singularities.
\newblock In {\em Knots, braids and singularities ({P}lans-sur-{B}ex, 1982)},
  volume~31 of {\em Monogr. Enseign. Math.}, pages 223--232. Enseignement
  Math., Geneva, 1983.

\bibitem{sagemath}
{The Sage Developers}.
\newblock {\em {S}ageMath, the {S}age {M}athematics {S}oftware {S}ystem
  ({V}ersion 9.2-3)}, 2020.
\newblock {\tt https://www.sagemath.org}.

\bibitem{Tristram}
A.~Tristram.
\newblock Some cobordism invariants for links.
\newblock {\em Proc. Cambridge Philos. Soc.}, 66:251--264, 1969.

\bibitem{Turaev}
V.~Turaev.
\newblock Reidemeister torsion in knot theory.
\newblock {\em Uspekhi Mat. Nauk}, 41(1(247)):97--147, 240, 1986.

\end{thebibliography}

\end{document}